\documentclass[12pt]{amsart}
\usepackage{amssymb}
\usepackage{amsmath, amscd}
\usepackage{amsthm}
\usepackage{xcolor}

\newtheorem{theorem}{Theorem}[section]

\newtheorem{proposition}[theorem]{Proposition}
\newtheorem{lemma}[theorem]{Lemma}

\newtheorem{corollary}[theorem]{Corollary}
\theoremstyle{definition}

\newtheorem{example}[theorem]{Example}

\newtheorem{problem}[theorem]{Problem}

\def\sgcb{semi-generalized co-Bassian}

\def\val#1{\vert #1 \vert}

\usepackage[utf8]{inputenc}
\usepackage{amssymb}
\usepackage{amsmath, amscd}
\usepackage{amsthm}
\usepackage{xcolor}

\begin{document}

\author[A.R. Chekhlov]{Andrey R. Chekhlov}
\address{Department of Mathematics and Mechanics, Regional Scientific and Educational Mathematical Center, Tomsk State University, 634050 Tomsk, Russia}
\email{cheklov@math.tsu.ru; a.r.che@yandex.ru}
\author[P.V. Danchev]{Peter V. Danchev}
\address{Institute of Mathematics and Informatics, Bulgarian Academy of Sciences, 1113 Sofia, Bulgaria}
\email{danchev@math.bas.bg; pvdanchev@yahoo.com}
\author[P.W. Keef]{Patrick W. Keef}
\address{Department of Mathematics, Whitman College, Walla Walla, WA 99362, USA}
\email{keef@whitman.edu}

\title[Semi-generalized co-Bassian groups] {Semi-generalized co-Bassian groups}

\maketitle

\begin{abstract} As a common non-trivial generalization of the notion of a generalized co-Bassian group, recently defined by the third author, we introduce the notion of a {\it semi-generalized co-Bassian} group and initiate its comprehensive study. Specifically, we give a complete characterization of these groups in the cases of $p$-torsion groups and groups of finite torsion-free rank by showing that these groups can be completely determined in terms of generalized finite $p$-ranks and also depends on their quotients modulo the maximal torsion subgroup. Surprisingly, for $p$-primary groups, the concept of a semi-generalized co-Bassian group is closely related to that of a generalized co-Bassian group.
\end{abstract}

\vskip2.0pc

\section{Introduction and Motivation}
	
Throughout the rest of the paper, unless specified something else, all groups will be additively written Abelian groups. We will primarily use the notation and terminology of \cite{F0,F1,F2}, but we will follow somewhat those from \cite{K} and \cite{Gr} as well. We just recall that an arbitrary subgroup $H$ of a group $G$ is {\it essential} in $G$ if, for any non-zero subgroup $S$ of $G$, the intersection between $H$ and $S$ is also non-zero. It is an elementary exercise to see that every subgroup will be essential as a subgroup of itself and thus, in particular, the zero subgroup $\{0\}$ is essential in itself too.

We begin with a brief review of some of the most important concepts which motivate our writing of the present article.

Mimicking \cite{CDG1}, a group $G$ is said to be {\it Bassian} if the existence of an injective homomorphism $\phi: G \to G/N$ for some subgroup $N$ of $G$ implies that $N = \{0\}$. More generally, imitating \cite{CDG2}, if this injection $\phi$ implies that $N$ is a direct summand of $G$, then $G$ is said to be {\it generalized Bassian}.

Note that Bassian groups were completely characterized in \cite{CDG1}. A crucial example of a generalized Bassian group which is {\it not} Bassian is the infinite {\it elementary $p$-group} that is an arbitrary infinite direct sum of the cyclic group $\mathbb{Z}_p$ for some fixed prime $p$. Unfortunately, the class of generalized Bassian groups is {\it not} fully characterized due to the unsettled at this stage problem of whether or {\it not} a subgroup of a generalized Bassian group remains so. On this vein, for some other interesting properties of subgroups of (generalized) Bassian groups, we refer the interested readers to \cite{DG}.

Taking into account all of the information we have so far, and in order to refine the group property of being generalized Bassian, we introduced in \cite{CDK} a group $G$ to be {\it semi-generalized Bassian} if, for any its subgroup $N$, the injective homomorphism $\phi: G \to G/N$ implies that $N$ is essential in a direct summand of $G$. Clearly, any generalized Bassian group is semi-generalized Bassian. In difference to the case of generalized Bassian groups, it is quite curious that semi-generalized Bassian groups were totally classified in the cases where $G$ is {\it not} truly mixed. In fact, we succeeded to classify torsion, torsion-free and splitting mixed semi-generalized Bassian groups.

Reciprocally, in \cite{Ke}, was defined the so-called {\it (generalized) co-Bassian} groups in the following manner: A group $G$ is termed co-Bassian if, for all subgroups $H\leq G$, whenever $\varphi: G\to G/H$ is an injective homomorphism, then $\varphi(G)=G/H$. In general, if $\varphi(G)$ is a direct summand of $G/H$, the group $G$ is termed generalized co-Bassian. Fortunately, these two classes of groups were completely described.

Furthermore, expanding generalized co-Bassian groups in the way of semi-generalized Bassian groups, we shall say that a group $G$ is {\it semi-generalized co-Bassian} if, for all $H\leq G$, the injection $\varphi: G\to G/H$ forces that $\varphi(G)$ is essential in a direct summand of $G/H$.

\medskip

Our motivating tool is to explore the structure of the so-defined semi-generalized co-Bassian groups and to show that their full description is rather more complicated than the case of generalized-Bassian groups. The essence of our claim is the situation of having such groups with infinite torsion-free rank and, especially, the lack of workable idea and approach how to prove or even to disprove that if $G$ is a group whose torsion subgroup $T$ is a direct sum of a divisible group and an elementary group such that the quotient $G/T$ is divisible, then $G$ is semi-generalized co-Bassian group or {\it not}.

\medskip

Our next work is structured thus: In the present first section, we gave a short retrospection of the basic notions. In the subsequent second section, we formulate our chief results and provide their complete proofs. Precisely, the paper's main goal is to characterize in detail semi-generalized co-Bassian $p$-groups as well as semi-generalized co-Bassian groups having finite torsion-free rank. Concretely, the most important of them state like this: {\it A $p$-group $G$ is semi-generalized co-Bassian if, and only if, its subgroup $pG$ has generalized finite $p$-rank} (see Theorem~\ref{major}); {\it Suppose $G$ is a group of finite torsion-free rank. Then, $G$ is semi-generalized co-Bassian if, and only if, for each prime $p$, the $p$-component of torsion $T_p$ is semi-generalized co-Bassian group such that either (a) $T_p$ possesses finite $p$-rank, or (b) $T_p$ is divisible, or (c) the quotient-group $G/T$ is $p$-divisible} (see Theorem~\ref{finiterank}). We also examine the case of groups with infinite torsion-free rank proving that {\it If a group $G$ is semi-generalized co-Bassian of infinite torsion-free rank, then the factor-group $G/T$ is divisible and the torsion part $T$ is a direct sum of a divisible group and an   elementary group} (see Proposition~\ref{4}). Finally, we end our work with certain comments and also state several open problems which quite logically arise and, hopefully, will stimulate a further exploration on the subject.

\section{Main Results and Their Proofs}

We first begin with the following elementary but useful observation, which was {\it not} stated and proved in \cite{Ke}, that shows to what extent the classes of co-Bassian and generalized co-Bassian groups differ each other.

\begin{proposition} If $G$ is a Hopfian group, then $G$ is generalized co-Bassian if, and only if, $G$ is co-Bassian.
\end{proposition}

\begin{proof} One direction being trivial, we deal with the opposite one. In fact, assume in a way of a contradiction that $G$ is generalized co-Bassian but, however, not co-Bassian. Then, there exists a subgroup $N$ of $G$ and an embedding $f: G\to G/N$. Since $G$ is generalized co-Bassian, one has that $f(G)$ is a direct summand of $G/N$, say
$G/N = f(G)\oplus H/N$ for some $H/N\neq\{\overline{0}\}$. Now, there is an isomorphism $$G/H\cong (G/N)/(H/N)\cong f(G),$$ say $\varphi:G/H\to f(G)$ and, moreover, if $\pi: G\to G/H$ is the canonical epimorphism, then the composition
$\varphi^{-1}\pi$ is obviously an epimorphism but a non-isomorphism of $G$ as $\ker (\varphi^{-1}\pi)=H\neq\{0\}$, contrary to the condition that $G$ is Hopfian. So, $G$ is really co-Bassian, as required.
\end{proof}

Our next incidental assertion curiously states as follows.

\begin{proposition} The free group $G$ is semi-generalized co-Bassian if, and only if, $G$ is Bassian.
\end{proposition}

\begin{proof} If we assume for a moment that $G$ has infinite rank, then there exists a subgroup $N\leq G$ such that the quotient-group $G/N$ is a torsion-free indecomposable group with ${\rm rank}(G)={\rm rank}(G/N)$. Thus, there is an injection $f:G\to G/N$ such that $f(G)$ is {\it not} essential in $G/N$, whence for the purification of $f(G)$ we have $\langle f(G)\rangle_*\neq G/N$ and, consequently, in view of indecomposability of $G/N$ the group $G$ is {\it not} semi-generalized co-Bassian, as expected. Therefore, $G$ has a finite rank and hence applying \cite{CDG1} the group $G$ must be Bassian, as claimed.
\end{proof}

We continue our work with two critical constructions as follows:

\begin{example}\label{(1)} (i) Every divisible group is semi-generalized co-Bassian.

\medskip

(ii) The group $G=\bigoplus_{\alpha}\mathbb{Z}_{p^n}$ is semi-generalized co-Bassian for all ordinals $\alpha$ and integers $n\geq 1$.
\end{example}

\begin{proof} (i) This follows automatically, because the factor-group of a divisible group is also divisible and,
for each divisible group $D$, any its subgroup is essential in some direct summand of $D$ (see \cite[Theorem~24.4]{F1}), as needed.

(ii) Since, for every injection $f: G\to G/H$, the image $f(G)$ is a direct sum of cyclic groups of order exactly $p^n$ and $G/H$ is a direct sum of cyclic groups of order $\leq p^n$, we can infer that $f(G)$ is a direct summand in $G/H$ (see \cite[Proposition~27.1]{F1}), as required.
\end{proof}

The following technicality is helpful for our presentation below.

\begin{lemma}\label{summands} The class of semi-generalized co-Bassian groups is closed under taking direct summands.
\end{lemma}

\begin{proof} Suppose $G$ is a semi-generalized co-Bassian group and $H$ is a summand of $G$; say $G=H\oplus G'$ for some $G'\leq G$. Assume now $N$ is a subgroup of $H$ and $f:H\to H/N:= \overline H$ is an injective homomorphism. Extending $f$ to
$f': G\to G/N:= \overline G\cong \overline H\oplus G'$ by letting it equal to the identity on $G'$, it remains an injective homomorphism.

We, therefore, have $\overline G= K \oplus A$, with $A$ containing $f'(G)=f(H)\oplus G'$ as an essential subgroup.  It readily follows that $A\cap \overline H\subseteq \overline H$ contains $f(H)$ as an essential subgroup. And since $A=(A\cap \overline H)\oplus G'$, we can also conclude that $A\cap \overline H$ is a summand of $\overline G$, and hence of $\overline H$, as required.
\end{proof}

We now continue with a series of statements necessary for the successful establishment of our two main results presented in what follows.

%\begin{proposition}\label{unbounded} A reduced unbounded $p$-group $G$ is not a semi-generalized co-Bassian group.
%\end{proposition}

%\begin{proof} Assume the contrary that $G$ is such a group. We may write $G=G_1\oplus G_2$, where $G_1$ is a bounded group and $\mathrm{finrank}(G_2)=r(G_2)$. In view of Lemma~\ref{summands}, we may also consider that $G_1$ is cyclic.
%Note that there exists a surjection $G_2\to D$, where $D$ is the divisible hull of $G_2$ (see, for instance, \cite{F1}). Now, let $G_1=\langle a\rangle$ such that ${\rm order}(a)=p^n$, $1<n\in\mathbb{N}$, $D_1\cong\mathbb{Z}_{p^{\infty}}$, $D=D_1\oplus D_2$ and $z\in D_1$ with ${\rm order}(z)=p^{n+1}$. Set $x=a+z$.
%Then, there exists an injection $f:G\to\langle x\rangle\oplus D_2$, where $f(G_1)\leq \langle x\rangle$ and
%$f(G_2)\leq D_2$. However, since $$\langle x\rangle\oplus D_2\leq\overline{G}=\langle a\rangle\oplus D,$$ we deduce that $f(G)$ is essential in some direct summand $B$ of $\overline{G}=B\oplus C$.
%Moreover, as $$B[p]=f(G)[p]=(\langle x\rangle\oplus D_2)[p]=D[p],$$ we derive that $B\cong D$ (see \cite[Lemma 66.1]{F1}). But since $px\in f(\langle a\rangle)$, we get $px\in B$ and, consequently, $|px|_{\overline{G}}=\infty$,
%thus contradicting the facts that $px=pa+pz$ and $|px|_{\overline{G}}=|pa|_{\overline{G}}=1$, as expected.
%\end{proof}

%This can be slightly extended in the following manner.

Before doing that, we need to review some notation, used in \cite {Ke}:  As above, when $G$ is some group, $T$ will always be its torsion subgroup and $T_p$ the $p$-component of $T$. If we have occasion to refer to some other group $A$, we will denote its torsion subgroup by $T_A$. For the group $G$, suppose $\phi,$ and $\pi$ are, respectively, injective and surjective homomorphism $G\to \overline G$; when necessary, we may assume $\overline G=G/N$ for some $N\leq G$ and $\pi$ is the usual epimorphism. if $A$ is any subgroup of $G$, we will let $\hat A=\phi(A)$ and $\overline A=\pi(A)$.

\begin{lemma}\label{1}
If $p$ is a prime and $T_p$ is a reduced unbounded $p$-group, then $G$ is not a semi-generalized co-Bassian group. In particular, any reduced unbounded $p$-group $G$ is not a semi-generalized co-Bassian group.
\end{lemma}

\begin{proof} Suppose $T_p=\langle b\rangle\oplus T'_b$, where $b$ has order at least $p^2$, and $G= \langle b\rangle\oplus G'$, where $T'_p\subseteq G'$. If $D$ is a divisible hull for $T'_p$ and $Z:=\mathbb Z_{p^\infty}$, then there is a surjective homomorphism $T_p'\to Z\oplus D$ which extends to a surjective homomorphism $\tau:G'\to Z\oplus D$. This, in turn, extends to a surjective homomorphism
$$\pi:G=\langle b\rangle\oplus G'\to  \overline G:=\langle b\rangle\oplus Z\oplus D\oplus (G'/T'_p)$$
by setting it equal to the identity on $\langle b\rangle$ and, for $x\in G'$, letting $\pi(x) = \tau (x)+(x+T'_p)$.

Let $c\in Z$ have the same order as $b$.   The inclusion $T'_p\subseteq D$ extends to a homomorphism $\sigma:G'\to D$.
Define a homomorphism $\phi:G\to \overline G$ as follows: If $x\in G'$, let $\phi(x) = \sigma(x) + (x+T'_p)$, and $\phi(b) =pb+c$.  It is readily checked that $\phi$ is injective.

If $G$ were \sgcb, it would follow that $\hat G=\phi(G)$ is essential in a summand $S$ of $\overline G$. Since every element of $\hat G[p]=Z[p]\oplus D[p]$ has infinite height in $\overline G$, the $p$-torsion subgroup of $S$ is a divisible subgroup of $\overline G$. However, $\phi(b)=pb+c$ is such a $p$-torsion element of $S$ and its $p$-height satisfies $\val {pb+c}_{S}=\val {pb+c}_{\overline G}=1$. With this contradiction, we can deduce that $G$ is not \sgcb, as stated.

The second part is immediate.
\end{proof}

%The next applicable statement appeared in \cite[Theorem~2.6]{CDK}.

%\begin{theorem}\label{6} The reduced $p$-group $G$ has the property that every subgroup is an essential subgroup of a summand of $G$ if, and only if, for some positive integer $n$, we have
%$$
%G\cong \left (\bigoplus_I \mathbb Z_{p^n}\right) \oplus \left (\bigoplus_J \mathbb Z_{p^{n+1}}\right);
%$$
%that is, $f_G(\alpha)=0$ unless $\alpha=n-1,n$.
%\end{theorem}

\begin{proposition}\label{leadin} Suppose $m>1$ is an integer and $\alpha$ is either $\infty$ or a positive integer with $m<\alpha$. If $G$ is a group with a summand of the form $\mathbb Z_{p^m}\oplus \mathbb Z_{p^\alpha}^{(\kappa)}$, where $\kappa$ is infinite, then $G$ is not semi-generalized Bassian.
\end{proposition}

\begin{proof} Utilizing Lemma~\ref{summands}, there is no loss of generality in assuming that $G=\mathbb Z_{p^m}\oplus \mathbb Z_{p^\alpha}^{(\kappa)}$. Suppose the first summand is $\langle b\rangle$. There is clearly a decomposition $\mathbb Z_{p^\alpha}^{(\kappa)}=\mathbb Z_{p^\alpha}\oplus Z$ with an isomorphism $\sigma:\mathbb Z_{p^\alpha}^{(\kappa)}\to Z$. Let $c$ be an element of the first term of this decomposition of order $p^m$; in particular, $pc\ne 0$ and $\val {pc}_p>1$.

Let $N=\langle pb\rangle$ and
$$\pi:G\to  \overline G:= (\langle b\rangle/N)\oplus \mathbb Z_{p^\alpha}^{(\kappa)}=(\langle b\rangle/N)\oplus\mathbb Z_{p^\alpha}\oplus Z$$
be the obvious surjection.

Define now the map $\phi: G\to \overline G$ as follows: $\phi(b) = (b+N)+c$ and $\phi$ restricted to $\mathbb Z_{p^\alpha}$ agrees with $\sigma$. It is readily checked that $\phi$ is injective. If, however, $\hat G$ were an essential subgroup of a summand of $\overline G$, it would easily follow that $\langle \phi(b)\rangle$ would be an essential subgroup of a summand of $(\langle b\rangle/N) \oplus  \mathbb Z_{p^\alpha}$. Since this summand would have to be cyclic, this cannot be true, because $\val {\phi(b)}_p=0$, $p\phi(b)=pc\ne 0$ and $\val {p\phi(b)}_p=\val {pc}_p>1$, as required.
\end{proof}

Following \cite{Ke}, the $p$-group $G$ is said to have \emph{generalized finite $p$-rank} if
$$G\cong\mathbb{Z}^{(\rho_1)}_{p^{\sigma_1}}\oplus\dots\oplus \mathbb{Z}^{(\rho_n)}_{p^{\sigma_n}},$$
where all ordinals of the increasing sequence $\sigma_1<\dots <\sigma_n$ are in $\omega\cup\{\infty\}$, and
$\rho_1,\dots,\rho_n$ are cardinals with $\rho_j$ finite whenever $j > 1$.

\medskip

We are now prepared to prove our first major assertion, which sounds quite surprising.

\begin{theorem}\label{major} The $p$-group $G$ is semi-generalized co-Bassian if, and only if, $pG$ is generalized co-Bassian, i.e., $pG$ has generalized finite $p$-rank.
\end{theorem}

\begin{proof} It is easily verified that $G$ satisfies this condition exactly when there is a decomposition $G=E\oplus H$, where $E$ is a $p$-high subgroup of $G$ and $H$ is a group that has generalized finite $p$-rank with no elementary summands.

Firstly, suppose $G$ is semi-generalized co-Bassian. Using Lemmas~\ref{summands} and \ref{1}, $G$ must be of the form:
$$
       G=\mathbb Z_{p^{\alpha_1}}^{(\kappa_1)}\oplus \mathbb Z_{p^{\alpha_2}}^{(\kappa_2)} \oplus \cdots \oplus \mathbb Z_{p^{\alpha_k}}^{(\kappa_k)},
$$
where the $\kappa$s are all (non-zero) cardinals and $\alpha_1<\alpha_2<\cdots <\alpha_k$ are either positive integers or $\infty$. The result, therefore, follows directly from the above Proposition~\ref{leadin}.

Conversely, suppose that  $G=E\oplus H$ possesses the above form. By definition, $H\cong \mathbb Z_{p^\alpha}^{(\kappa)} \oplus F$, where $\alpha>1$ is either an integer or $\infty$, $\kappa$ is a cardinal, $F$ has finite $p$-rank and $F[p]=p^\alpha H[p]$ when $\alpha$ is an integer and $F=0$ when $\alpha=\infty$. If  $\kappa$ is finite, then $G$ itself is generalized co-Bassian, and hence semi-generalized co-Bassian, as desired. So, we may assume hereafter that $\kappa$ is infinite.

Let $N$ be a subgroup of $G$, and set $\overline G:=G/N$ and $\phi: G\to \overline G$ is an injective homomorphism. As in \cite K, for any subgroup $A$ of $G$, we let $\hat A=\phi(A)$ and $\overline A$ be the image of $A$ under the canonical homomorphism $\pi:G\to \overline G$.  We need to show $\hat G=\phi(G)$ is an essential subgroup of a summand of $\overline G$.

To that goal, suppose first that $\alpha=\infty$; so, $H=\mathbb Z_{p^\infty}^{(\kappa)}$ is the maximal divisible subgroup of $G$ and $F=0$. Thus, $\overline H$ is divisible and $p(\overline G/\overline H)=0$, so that $\overline G=E'\oplus \overline H$, where $E'$ is also elementary.

Clearly, $\hat H$ will also be divisible, so that $\overline H=D\oplus \hat H$, where $D$ is divisible. If $$E'':=\hat G\cap (E'\oplus D)\subseteq (E'\oplus D)[p],$$ it follows that $\hat G=E''\oplus \hat H$. But it also plainly follows that any subsocle of a direct sum of an elementary group and a divisible group is supported by a summand of the containing group (see \cite{F1} as well). In particular, this means that $\hat G=E''\oplus \hat H$ is an essential subgroup of a summand of $G$, as required.

The proof where $\alpha$ is a positive integer is similar; for clarity, denote $\alpha$ by $n$. Since $p^nG=p^nH$ is co-Bassian and $\phi$ and $\pi$ restrict to an injection and surjection, respectively, on $p^nG$, it follows that $p^n \hat H=p^n\hat G=p^n\overline G$. Since $H$ has no summands isomorphic to $\mathbb Z_{p^m}$ for $m<n$, this readily implies that $\overline G=\hat H\oplus K$, where $p^n K=0$. It follows that $\hat G\cap K\subseteq K[p]$. And since $K$ is bounded, $\hat G\cap K$ supports a summand $K'$ of $K$. Consequently, $K'\oplus \hat H$ will be a summand of $\overline G$ containing $\hat G$ as an essential subgroup, as wanted.
\end{proof}

The pivotal instrument, needed to establish our second major assertion listed below, is the following.

\begin{proposition}\label{3} If $G$ is a \sgcb\ group and $p$ is a prime, then $T_p$ is \sgcb.
\end{proposition}

\begin{proof} Suppose $G$ is a \sgcb\ group, $p$ is a prime and $T_p=R\oplus D$, where $R$ is reduced and $D$ is divisible. If $D$ has infinite $p$-rank, then by Proposition~\ref{leadin} and Lemma~\ref{summands}, $R=\{0\}$ so that $T_p$ is \sgcb. So, we may assume that $D$ has finite rank. If $G=G'\oplus D$, where $R=T_{G'}$, then by Lemma~\ref{summands}, $G'$ is \sgcb, which by Lemma~\ref{1} implies that $R$ is bounded. The fact that $T_p$ is \sgcb\ then quickly follows from Theorem~\ref{major}, Proposition~\ref{leadin} and Lemma~\ref{summands}.
\end{proof}

We next present a full characterization of the \sgcb\ group of finite torsion-free rank.

\begin{theorem}\label{finiterank} If $G$ has finite torsion-free rank, then $G$ is \sgcb\ if, and only if, for every prime $p$, $T_p$ is \sgcb\ and either (a) $T_p$ has finite $p$-rank, or (b) $T_p$ is divisible, or (c) $G/T$ is $p$-divisible.
\end{theorem}

\begin{proof} Suppose $G$ is \sgcb\ and $p$ is a prime. By Proposition~\ref{3}, $T_p$ is \sgcb. Assume $G/T$ is not $p$-divisible and $T_p$ has infinite $p$-rank; we need to show that $T_p$ is divisible. Assume not, so that it has a non-zero cyclic summand $C=\langle c\rangle$ of minimal order $p^k$. It is straightforward to verify that there is a decomposition $T_p=C\oplus K$ and an injection $\phi:T_p\to K$ such that $\phi(T_p[p])=K[p]$.

Since $T_p$ is the direct sum of a bounded and a divisible group, there is a decomposition $G=T_p\oplus A=C\oplus K\oplus A$; we are therefore assuming that multiplication by $p$ is an injective, but not surjective, endomorphism on $A$. It follows that $A/p^k A$ is isomorphic to a direct sum of copies of $\mathbb Z_{p^k}$. So, if $a\in A$ has $p$-height 0, then $a+p^k A$ will have order $p^k$ in $A/p^k A$, and hence will generate a summand. Thus, there is a composite of natural homomorphisms
$$\gamma:A\to A/p^k A\to \langle a+p^k A\rangle \to C$$ such that $\gamma(a)=c$. Extend $\phi$ to an injection $\phi:G\to G$ by setting, for all $x\in A$,  $\phi(x)=\gamma(x)+px$.

Letting $\pi:G\to \overline G:=G$ be the identity, then since $G$ is \sgcb, we can conclude that $\hat G=\phi(G)$ is an essential subgroup of some summand $S$ of $G$. Since $\hat G[p] =K[p]$, $p^{k-1}c$ does not represent an element of the Ulm factor $U_{k-1}(S)\subseteq U_{k-1}(G)$. On the other hand, considering $p$-heights, $$\val {p^{k-1} \phi(a)}=\val {p^{k-1}c+p^k a}=k-1$$ and $$\val {p^k \phi(a)}=\val {p^{k+1} a}=k+1,$$ so  there is an element $s\in S$ such that $\val s=k$ and $ps=p^k \phi(a)$. Therefore, $p^{k-1}\phi(a)-s$ does represent an element of $U_{k-1}(S)$. But since
$$\val {(p^{k-1}\phi(a)-s)-p^{k-1} c}=\val {p^k a-s}\geq k,$$ we can conclude that $p^{k-1}\phi(a)-s$ and $p^{k-1}c$ represent the same element of $U_{k-1}(G)$, giving a contradiction. Therefore, $G$ must be divisible, completing the proof of sufficiency.

\smallskip

Conversely, suppose that for all primes $p$, $T_p$ is \sgcb\ and if $G/T$ is not $p$-divisible, then either $G$ has finite $p$-rank or $T_p$ is divisible. To show $G$ is \sgcb, suppose $\pi$ and $\phi$ are, respectively, surjective and injective homomorphisms $G\to \overline G$.

Let $\mathcal P_1$ be the collection of primes such that $G/T$ is $p$-divisible, and $\mathcal P_2$ be those primes $p\not\in \mathcal P_1$ for which $T_p$ has finite $p$-rank and $\mathcal P_3$ be the remaining primes; so if $p\in \mathcal P_3$, then $G/T$ is not $p$-divisible and $T_p$ is divisible (of infinite rank).

If $N$ is the kernel of $\pi$, then the fact that $G$ has finite torsion-free rank implies that $N\subseteq T$ and $T_{\overline G}=\overline T$. Therefore, for each prime $p$, $\pi$ and $\phi$ restrict to surjective and injective homomorphisms $T_p\to \overline T_p$. Since we are assuming $T_p$ is \sgcb, we can conclude that there is a decomposition $\overline T_p=L_p\oplus B_p$, where $L_p$ contains $\hat T_p$ as an essential subgroup. Note that if $p\in \mathcal P_2$, then $L_p = \hat T_p=\overline T_p$ and $B_p=0$; and if $p\in \mathcal P_3$, then both $L_p=\hat T_p$ and $B_p$ are divisible. Letting $L=\oplus_p L_p$ and $B=\oplus_p B_p$, then $\hat T$ is essential in $L$ and $\overline T=L\oplus B$.

Observe that there are isomorphisms of torsion-free finite rank groups $\hat G/\hat T\cong G/T\cong \overline G/\overline T$. Therefore, since $\hat G/\hat T$ clearly embeds in $\overline G/\overline T$ and they are isomorphic,
$$
     \overline G/[\overline T+\hat G]\cong (\overline G/\overline T)/(\hat G/\hat T):=X  \eqno {(\dag)}
$$
is a finite group.  And since both groups in this quotient are $p$-divisible for all primes $p\in \mathcal P_1$, it follows that $X$ has no $p$-torsion whenever $p\in \mathcal P_1$.

Suppose $x_0, \dots x_k$ in $ \overline G$ project onto a linearly independent generating set for $X$. If $n_i$ is the order of the image of $x_i$ in $X$, then $n_i$ is not divisible by any prime from $\mathcal P_1$. So, there is a finite set of primes $F\subseteq \mathcal P_2\cup \mathcal P_3$, such that each $n_i$ is only divisible by primes from $F$.

For each $i$, there is a $y_i\in \hat G$ and $z_i\in \overline T$ such that $n_ix_i=y_i+z_i$.  For each $i$ there is an $m_i$ not divisible by any prime of $F$ such that $m_iz_i$ has order divisible only by primes in $F$. Clearly, the $m_ix_i$ still project to a linearly independent generating set for $X$, so replacing $x_i, y_i$ and $z_i$ by $m_ix_i,m_i y_i$ and $m_iz_i$ respectively, we may assume that each $z_i$ is in $\oplus_{p\in F} \overline T_p= \oplus_{p\in F} (\hat T_p\oplus B_p)$. We may clearly absorb each component of $z_i$ in this decomposition that lives in $\hat T_p$ into the corresponding element $y_i\in \hat G$. So, we may assume $z_i\in \oplus_{p\in F}  B_p$. But since $F\subseteq   \mathcal P_2\cup \mathcal P_3$, the direct sum $\oplus_{p\in F}  B_p$ is divisible. Let $w_i\in \oplus_{p\in F}  B_p$ satisfy $n_i w_i = z_i.$  So, $x_i':=x_i-w_i$ represent the exact same elements of $X$, but now, $n_i x_i'=y_i$.

Replacing $x_i$ by $x_i'$ and letting $$S=\hat G+\langle x_1, \dots, x_k\rangle+ L\subseteq \overline G,$$ we claim that $T_S=L$: The containment $\supseteq$ being obvious, suppose $s\in T_S$; we need to show $s\in L$. By definition, there is $u\in \hat G$, $v\in L$ and integers $j_0, \dots, j_k$ such that $$s= u+j_0x_0+\cdots + j_kx_k+v.$$ If we map this relation into the quotient $X$, if readily follows that for each $i$, that  $j_i x_i\in \hat G$. Therefore, $$u+j_0x_0+\cdots + j_kx_k=s-v\in \hat G\cap \overline T=\hat T\subseteq L.$$ This gives $s=v+(s-v)\in L$, as required.

Notice also that $(\dag)$ also implies that $\overline G=S+\overline T$.
We have $\overline T= L\oplus B=T_S\oplus B$, so that $S+B=S+\overline T=G$ and $S\cap B=T_S\cap B=0$. Therefore, $\overline G=S\oplus B$.

And finally, since $\hat T$ is essential in $L=T_S$ and $S/\hat G$ is torsion, it readily follows that $\hat G$ is an essential subgroup of $S$, as required.
\end{proof}

The next result follows immediately from Theorem~\ref{finiterank}.

\begin{corollary}\label{7} The finite torsion-free rank group $G$ is \sgcb\ if, and only if, for every prime $p$, its localization $G_{(p)}$ is \sgcb.
\end{corollary}

Having characterized the \sgcb\ groups of finite rank, we turn to a consideration of those with infinite rank. We will say a group is {\it D+E} if it is isomorphic to $D\oplus E$, where $D$ is divisible and $E$ is elementary.

\begin{proposition}\label{4} If $G$ is \sgcb\ of infinite torsion-free rank, then $G/T$ is divisible and $T$ is {\rm D+E}.
\end{proposition}

\begin{proof} Let $\kappa$ be the rank of $G$. Since each $T_p$ is \sgcb, $T$ will be a direct sum of cocyclic groups.  This easily implies that there is a decomposition $G=G'\oplus T'$, where $T'$ is torsion, $G'$ has cardinality $\kappa$ and $T$ is D+E if, and only if, that condition holds for $T_{G'}$. So, by Lemma~\ref{summands}, there is no loss of generality in assuming that $G$ has cardinality and rank $\kappa$.

Suppose first that $G/T$ is not $p$-divisible for some $p$. Let $A$ be a subgroup of $G$ containing $T$ such that $G/A$ has order $p$ and let $\tau: G\to G/A$ be the usual epimorphism.  If $D$ is a divisible hull for $G$, then there is clearly a surjective homomorphism $A\to D$ which extends to a homomorphism $\gamma:G\to D$.

The homomorphism $\pi: G\to \overline G:=(G/A)\oplus D$ given by $\pi(x)=(x+A)+\gamma (x)$ is clearly onto. The homomorphism $\phi:G\to (G/A)\oplus D$  given by $\phi(x) = (x+A)+ x$ is clearly injective.

If $G$ were \sgcb, then $\hat G$ would be contained as an essential subgroup of a summand $S\subseteq \overline G$. Since $S[p]=0\oplus T[p]\subseteq  D\subseteq \overline G$, we could conclude that the $p$-torsion subgroup of $S$ is divisible. If $x\not\in A$, then $\val {\phi(x)}_S=\val {\phi(x)}_{\overline G}=0$, and $\val{p\phi(x)}_S=\val {p\phi(x)}_{\overline G}=\infty$, which cannot happen when the $p$-torsion of $S$ is divisible.

Suppose now that $G$ has a cyclic summand, $C=\langle c\rangle$, of order $p^k$, where $k>1$. Let $G=C\oplus A$. If $D$ is a divisible hull for $A$ and $Z\cong \mathbb Z_{p^\infty}$, then there is clearly a surjection $$\pi:A\to Z\oplus D$$ and an injection $$\phi:A\subseteq D\subseteq Z\oplus D.$$ Extending $\pi$ to $G$ by setting it equal to the identity on $C$ makes it a surjection $$G\to \overline G:=C\oplus Z\oplus D.$$ Choosing $z\in Z$ of order $p^k$ and extending $\phi$ by setting $\phi(c)=pc+z$ turns it into an injection $\phi:G\to \overline G$. Again $\hat G[p]\subseteq (Z\oplus D)[p]$, so that all elements of $\hat G[p]$ have infinite height. So, if $\hat G$ were contained as an essential subgroup of a summand $S$, then the $p$-torsion subgroup of $S$ would be divisible. This, however, contradicts the fact that $\val {\phi(c)}_{\overline G}=1\ne \infty$.
\end{proof}

As a direct consequence, we yield:

\begin{corollary} Suppose $G$ is a \sgcb\ group of infinite torsion-free rank and $G\cong D\oplus R$, where $D$ is divisible and $R$ is reduced. Then, $T_R$ is elementary and $R/T_R$ is divisible.
\end{corollary}

The following is a partial converse to Proposition~\ref{4}.

\begin{proposition}\label{5} If the group $G$ itself is {\rm D+E}, then $G$ is \sgcb.
\end{proposition}

\begin{proof} It is easily seen that a homomorphic image of a D+E-group is also D+E, and that any D+E-subgroup of a D+E-group will be essential in a summand. This will imply the pursued result.
\end{proof}

So, if $G$ is a torsion-splitting group, i.e., $T$ is a summand of $G$, then the converse of Proposition~\ref{4} holds. However, it is not clear that this converse holds when $G$ is not torsion-splitting. In fact, the problem may be quite difficult and we leave it unsettled at this stage. 

Thus, we explicitly pose the following:

\begin{problem} Does it follow that a group $G$ is \sgcb, provided its torsion part $T$ is a direct sum of a divisible group and an elementary group and the quotient $G/T$ is a divisible group?
\end{problem}

\medskip
\medskip

\noindent {\bf Funding:} The work of the first-named author A.R. Chekhlov was supported by the Ministry of Science and Higher Education of Russia (agreement No. 075-02-2023-943). The work of the second-named author P.V. Danchev was partially supported by the Bulgarian National Science Fund under Grant KP-06 No. 32/1 of December 07, 2019, as well as by the Junta de Andaluc\'ia under Grant FQM 264, and by the BIDEB 2221 of T\"UB\'ITAK.

\end{document}